\documentclass{gNST2e}

\RequirePackage[OT1]{fontenc}
\usepackage{amsfonts,amsmath,amssymb,amsthm,amscd,stmaryrd}
\usepackage{dsfont}
\usepackage{multirow}
\RequirePackage[colorlinks,citecolor=blue,urlcolor=blue]{hyperref}
\usepackage{graphicx}
\usepackage{tabularx}
\usepackage{booktabs}
\usepackage{verbatim}
\usepackage{tikz}

\numberwithin{equation}{section}
\theoremstyle{plain}
\newtheorem{thm}{Theorem}[section]

\newtheorem{lem}[thm]{Lemma}
\newtheorem{corrol}[thm]{Corollary}
\newtheorem{example}[thm]{Example}

\def \P {\mathbb{P}}
\def \E {\mathbb{E}}
\def \R {\mathbb{R}}
\def \V {\mathbb{V}}
\def \N {\mathbb{N}}
\def \C {\mathbb{C}\text{ov}}
\def \bnk {b\left(n/k\right)}
\def \bn2 {b\left(\frac{n}{k}\right)}
\def \ank {a\left(n/k\right)}
\def \an2 {a\left(\frac{n}{k}\right)}

\def \N {\mathbb{N}}
\def \a {\alpha}

\def \l {\lambda}
\def \d {\delta}
\def \e {\varepsilon}

\def \S {\mathcal{S}}

\def \x {\mathbf{x}}

\DeclareMathOperator*{\argmin}{argmin}

\begin{document}
\title{On Tail Index Estimation based on Multivariate Data} 

\maketitle

\author{A. Dematteo $^{\rm a}$ $^{\ast}$\thanks{$^\ast$Corresponding author. Email: antoine.dematteo@telecom-paristech.fr}
\vspace{6pt} and S. Cl\'{e}men\c{c}on$^{\rm b}$\\\vspace{6pt}  $^{a}${\em{Telecom ParisTech, 46 rue
Barrault, 75634 Paris cedex 13, France. \\
           \printead{e1,e2}}};
$^{b}${\em{GazTransport \& Technigaz (GTT), 1 route de Versailles, 78400 Saint-R\'emy-l\`es-Cheuvreuse, France}}\\}

\begin{abstract}
  This article is devoted to the study of tail index estimation based on i.i.d. multivariate observations, drawn from a standard heavy-tailed distribution, \textit{i.e.} of which Pareto-like marginals share the same tail index. A multivariate Central Limit Theorem for a random vector, whose components correspond to (possibly dependent) Hill estimators of the common tail index $\alpha$, is established under mild conditions. Motivated by the statistical analysis of extremal spatial data in particular, we introduce the concept of (standard) heavy-tailed random field of tail index $\alpha$ and show how this limit result can be used in order to build an estimator of $\alpha$ with small asymptotic mean squared error, through a proper convex linear combination of the coordinates. Beyond asymptotic results, simulation experiments illustrating the relevance of the approach promoted are also presented.
\end{abstract}

\begin{keywords} Heavy-tail modelling; multivariate statistics; tail index estimation;  Hill estimator; aggregation
\end{keywords}

\begin{classcode}62G32;91B30\end{classcode}

\section{Introduction}
It is the main purpose of this paper to provide a sound theoretical framework for risk assessment, when dangerous events coincide with the occurrence of extremal values of a random field with an intrinsic heavy-tail behaviour. The theory of regularly varying functions provides a semi-parametric framework, with the ability to give an appropriate description of heavy-tail phenomena. In risk assessment, it is the main mathematical tool to carry out worst-case risk analyses in various fields. It is widely used for risk quantification in Finance (\citet{RFM}), Insurance (\citet{mikosch}) or for the modelling of natural hazards, see \citet{Tawn92} or \citet{Coles94}. \citet{Hult05} introduce the regularly varying processes of $D\left([0,1],\R^d\right)$, the space of right-continuous functions from $[0,1]$ in $\R^d$ with left-limit. The present article consider random fields observed  on a lattice with an intrinsic marginal heavy-tail behaviour with tail index $\a>0$. The parameter $\a$ governing the extremal behaviour of the marginals of the random field, we consider the problem of estimating it. Whereas a variety of statistical techniques for tail index estimation have been proposed in the univariate setup (see Chap. 6 in \citet{ekm} for instance), focus is here on extension of the popular Hill inference method (see \citet{Hill1975}). Incidentally, we point out that the analysis carried out in this paper can be extended to alternative estimation procedures.

In the univariate i.i.d. case, several authors investigated the asymptotic normality of the Hill estimator under various assumptions, including  \citet{Davis1984}, \citet{Beirlant89} or \citet{Teugels85}. In a primary work, \citet{Hsing91} showed a central limit theorem in a weak dependent setting under suitable mixing and stationary conditions. Recently, these conditions have been considerably weakened in \citet{Hill10}.
 Here, the framework we consider is quite different. The data analysed correspond to i.i.d observations of a random field on a compact set $\mathcal{S}\subset \mathbb{R}^d$ with
  $d\geq 1$ and where each margin has the same tail index. Precisely, the random field is observed on a lattice $s_1,\; \ldots,\; s_l$: to each vertex of the lattice corresponds a sequence of $n\geq 1$ i.i.d. random variables with tail index $\a$, the collection of sequences being not independent.  Denoting by $H^{(i)}_{k_i,n}$ the Hill estimator using the $k_i$ largest observations at location $s_i$, $1\leq i\leq d$, the accuracy of the estimator $H_{k_i,n}^{(i)}$ is known to depend dramatically on $k_i$. There are several ways to choose this parameter, based on the \textit{Hill plot} ($k_i$ is picked in a region where the plot looks flat or on resampling procedures for instance, see \citet{Danielsson2001}).
Eventually, the \textit{optimal }$k_i$'s are likely to be different, depending highly on the location $s_i$. Here, we consider the issue of accurately estimating the parameter $\a$ based on the collection of estimators $H^{(1)}_{k_1,n},\; \ldots,\; H^{(l)}_{k_l,n}$ and investigate the advantage of suitably chosen convex linear combination of the local estimates over a simple uniform average. The study is based on a limit theorem established in this paper, claiming that $\sqrt{k_1}\left(H^{(1)}_{k_1,n}-1/\alpha,\dots,H^{(l)}_{k_l,n}-1/\alpha\right)$ is asymptotically Gaussian under mild assumptions, together with the computation of an estimate of the asymptotic covariance matrix. These  results can be used to derive the limit distribution of any linear combination of the local estimators and, as a byproduct, to find an optimal convex linear combination regarding the asymptotic mean squared error (AMSE). For illustration purpose, experimental results are also presented in this article, supporting the use of the approach promoted, for risk assessment in the shipping industry in particular.

The paper is organized as follows. In section \ref{background}, we briefly recall some theoretical results about the concept of regular variations. In section \ref{MainTh}, the main results of this paper are stated, establishing in particular the asymptotic normality of the multivariate statistic $H_n$ whose components coincide with local Hill estimators, and explaining how to derive a tail index estimator with minimum AMSE. The simulation results are provided in section \ref{simulations}. Numerical results are displayed in section \ref{sloshing}, while technical proofs are postponed to the Appendix section.

\section{Background and Preliminaries}
\label{background}

We start off with some background theory on regular variation and the measure of extremal dependence. Next, we  briefly recall the classical Hill approach to tail index estimation in the context of i.i.d. univariate data drawn from a heavy-tailed distribution. The indicator function of any event $\mathcal{E}$ is denoted by $\mathds{1}(\mathcal{E})$. For all $\mathbf{x}=(x_1,\;\ldots,\; x_l)\in \mathbb{R}^l$, the cartesian product $\prod_{i=1}^l(x_i,+\infty]$ is denoted by $(\mathbf{x},+\infty]$. In addition, all operations in what follows are meant component-wise, \textit{e.g.} $1/\mathbf{k}=(1/k_1,\; \ldots,\; 1/k_l)$ for any $\mathbf{k}=(k_1,\; \ldots,\; k_l)\in \mathbb{N}^{*l}$. The covariance  of two square integrable random variables $X$ and $Y$ is denoted by $\C[X,Y]$ and $\C[X,X]= \V[X]$.

\subsection{Heavy-tailed random variables}
By definition, heavy-tail phenomena are those which are ruled by very large values, occurring with a far from negligible probability and with significant impact on the system under study. When the phenomenon of interest is described by the distribution of a univariate r.v., the theory of regularly varying functions provides the appropriate mathematical framework for heavy-tail analysis. For clarity's sake and in order to introduce some notations which shall be widely used in the sequel, we recall some related theoretical background. One may refer to \citet{Resnick1987} for an excellent account of the theory of regularly varying functions and its application to heavy-tail analysis.
\par Let $\a>0$, we set
 $$\mathcal{RV}_{\a}=\{U:\R_+\rightarrow\R_+\text{ borelian } |\lim_{t\rightarrow\infty}\frac{U(tx)}{U(t)}=x^{\a},\quad \forall x > 0 \}$$
the space of regularly varying functions (at infinity) with tail-index $\a$. Let $X$ be a random variable with cumulative distribution function (\textit{cdf} in short) $F$ and survival function $\overline{F}=1-F$. It is said to belong to the set of random variables with a heavy tail of index $\a$ if $\overline{F}\in\mathcal{RV}_{-\a}$. In addition, the heavy-tail property can be classically formulated in terms of vague convergence to a homogeneous Radon measure. Indeed, the random variable $X$ is heavy-tailed with index $\a$ if and only if:
$$
n\mathbb{P}\left(X/F^{-1}(1-1/n)\in \cdot  \right) \xrightarrow[n\rightarrow\infty]{v} \nu_{\alpha}(\cdot) \text{ in } M_+(0,\, \infty],
$$
where $\xrightarrow[n\rightarrow\infty]{v}$ denotes vague convergence (the reader may refer to \citet[chap 3.]{resnick2007heavy} for further details), $F^{-1}(u)=\inf\{t:\; F(t)\geq u\}$ denotes $F$'s generalized inverse, $\nu_{\alpha}(x,\, \infty]=x^{-\alpha}$, $M_+(0,\, \infty]$ the set of nonnegative Radon measures on $[0,\, \infty]^d\backslash\{0\}$.
\par Based on this characterization, the heavy-tail model can be straightforwardly extended to the multivariate setup. Now, consider a $d$-dimensional random vector $\mathbf{X} = (X_1,\dots,X_d)$ taking its values in $\mathbb{R}_+^d$ and where each margin has the same tail index $\a$. It is said to be a  \textit{standard} heavy tailed random vector with tail index $\a>0$ if there exists a non null positive Radon measure $\nu$ on  $[0,\, \infty]^d\backslash\{0\}$ such that:
\begin{equation}
\label{RV}
x\P\left( \left(\frac{X_1}{a^{(i)}(x)},\dots,\frac{X_d}{a^{(d)}(x)}\right)\in \cdot\right) \xrightarrow[x\rightarrow\infty]{v} \nu(\cdot),
\end{equation}
where for $i=1\dots d$, $a^{(i)}:x\mapsto F_{i}^{-1}(1-1/x)$ and $F_i$ is the cdf of the $i$th component. In such a case, $\nu$ fulfils the homogeneity property $\nu(tC) = t^{-\a}\nu(C)$ for all $t>0$ and any Borel set $C$ of $[0,\, \infty]^d\backslash\{0\}$, and all components are tail equivalent: $1-F_i\in \mathcal{RV}_{-\a}$ for $i=1,\;\ldots,\; d$.

\subsection{The Hill method for tail index estimation}
A variety of estimators of $\a$ have been proposed in the statistical literature in the context of univariate i.i.d. observations drawn from a heavy-tailed distribution, see \citet{Hill1975, Pickands_1975, deHaan1980, Mason1982, Davis1984, Csorgo1985, Dekkers89}. In this paper, focus is on the popular Hill estimator but ideas to extend this work to other estimators are introduced in Appendix section \ref{extension}. The Hill estimator is defined as follows. Let $X_1,\; \ldots,\; X_n$ be observations drawn from a heavy-tailed probability distribution with tail index $\a$ and denote $X(1)\geq \ldots\geq X(n)$ the corresponding order statistics. Let $k$ such that $k\rightarrow\infty$ and $k/n\rightarrow 0$; the Hill estimator of the shape parameter $1/\a$ based on the $k$-largest observations is given by 

$$H_{k,n}=\frac{1}{k}\sum_{i=1}^{k}\log\left( \frac{X(i)}{X(k+1)}\right)=\int_{1}^{\infty}\frac{1}{k}\sum_{i=1}^{n}\mathds{1}\left(\frac{X_i}{X(k+1)}>x\right)\frac{dx}{x}.$$

The asymptotic behaviour of this estimator has been extensively investigated. Weak consistency is shown by \citet{Mason1982} for i.i.d. sequences, by \citet{Hsing91} for weakly dependent sequences, and by \citet{Starica1997} for linear processes. Strong consistency is proved in \citet{Deheuvels88} in the i.i.d. setup under the additional assumption that $k/\log\log n\rightarrow \infty$. The asymptotic normality with a deterministic centring by $1/\a$ requires additional assumptions on the distribution $F$ of $X$ and has been established in \citet{Teugels85,deHaan98,Geluk97,Hill10}. In this case,  
$$\sqrt{k}\left(H_{k,n}-1/\a\right)\Rightarrow \mathcal{N}\left(0,1/\a^2\right),$$
where $\Rightarrow$ means convergence in distribution, and $k$ is a function of $n$.  However, depending on the choice of $k$ and on the property of $F$ regarding second order regular variation, the Hill estimator can be significantly biased. This is studied for instance in \citet{deHaan98b}.

Hence, the practical issue of choosing $k$  is particularly important and has been addressed in various papers. They mostly rely on the second order regular variations and seek to achieve the best trade-off between bias and variance. \citet{Drees98} derive  a sequential estimator of the optimal $k$ that does not require prior knowledge of the second order parameters. In \citet{Danielsson2001} a subsamble bootstrap procedure is proposed, where the sample fraction that minimizes the asymptotic mean-squared error is adaptively determined. Graphical procedures are also available. In \citet{Drees2000} the popular Hill plot is compared to the AltHill Plot that is proved to be more accurate if $F$ is not strictly Pareto.

\section{Tail index estimation for a heavy-tailed random vector}
\label{MainTh}
Consider a heavy-tailed random vector $\mathbf{X}=\big(X_{1},\dots,X_{l}\big)$, $l>1$. Although it is not assumed that all the margins have the same distribution, we suppose that share the same (unknown) tail index $\alpha$ and that we have $n$ observations $\big((X_{1,i},\dots, X_{l,i})\big)_{i=1\dots n}$ of the vector $\mathbf{X}$.	

In order to state the main results of the paper, we introduce some additional notations. Denote respectively by $F$ and $\overline{F}$ the cdf and the survival function of the r.v. $\mathbf{X}=(X_1,\; \ldots,\; X_l)$ and by $F_i$ and $\overline{F}_i$ those of $X_i$, $i=1,\ldots,l$. Here and throughout, for $i\in\{1,\dots,l\}$, $X_i(1)>\ldots> X_i(n)$ are the order statistics related to the sample $(X_{i,1},\dots,X_{i,n})$ and $H^{(i)}_{k,n}$ is the Hill estimator based on the $k$-largest values observed at location $i$. The quantile of order $(n-k+1)/n$ of $F_i$ is denoted by $a^{(i)}(n/k)$ and we set $\mathbf{a}(n/k)=\big(a^{(1)}(n/k),\dots,a^{(l)}(n/k)\big)$. Finally, recall that there exists a Radon measure $\nu$ such that the following convergence holds true (see Chapter 6 in \citet{resnick2007heavy} for instance):
\begin{equation}
\label{exponent}
\frac{n}{k}\P\left(\frac{X_{1}}{a^{(1)}(n/k)}>x_1,\dots,\frac{X_{l}}{a^{(l)}(n/k)}>x_l\right)\xrightarrow[n\rightarrow\infty]{v} \nu\big(\left(\mathbf{x},+\infty\right]\big),
\end{equation}
where $\mathbf{x}=(x_1,\dots,x_l)\in\R_+^l$. In the sequel, $\nu\big(\mathbf{x}\big)$ will abusively stand  for $\nu\big(\left(\mathbf{x},+\infty\right]\big)$. We also set $\nu_{i,j}(x_i,x_j)$ as the limit of $\nu(x_1,\dots,x_l)$ when all the components but the $i-th$ and the $j-th$ tend to 0. We assume that all these limits exist for any choice of $i$ or $j$.

We point out that all the results of this section can be extended to alternative estimators of the tail index, such as those studied in \citet{Dekkers89} (Moment estimator) or \citet{Vries1996} ( Ratio of moments estimator). Technical details are deferred to Appendix section \ref{extension}.

It is the goal pursued in this section to show how to combine, in a linear and convex fashion, the local Hill estimators in order to refine the estimation of $\a$ in the AMSE sense.

\subsection{A multivariate functional central limit theorem}
As a first go, we start with recalling a result of \citet[Proposition 3]{deHaan93}, which can also be found in \citet[lemma 3.1]{Einmahl97}, stating the convergence of the tail empirical process toward a Gaussian process. This result is next used to prove a Central Limit Theorem with a random  centring for the random vector whose components correspond to the local Hill estimators, the latter being all viewed as functionals of the tail empirical process. Under some additional assumptions, the random centring is removed and replaced by $1/\a$. The case where the number of observations involved in the local Hill estimator components depends on the location considered is also dealt with. The main application of this result is that the Hill estimator can be replaced by an alternative estimator with a smaller asymptotic mean squared error.

\begin{thm}
\label{GaussField}{\sc (A Functional Central Limit Theorem)}
Equipped with the notations previously introduced, the following vague convergence (in the space of continuous functions from $\R_{+}^{d}$ to $\R$) holds : as $n,\; k\rightarrow +\infty$ such that $k=o(n)$, we have
\begin{equation}
\label{tailprocess}
\sqrt{k}\left(\frac{1}{k}\sum_{i=1}^{n}\mathds{1}\left(\frac{\mathbf{X}_i}{\mathbf{a}(n/k)}>\x\right)-\frac{n}{k}\overline{F}\left(\mathbf{a}(n/k)\x\right)\right) \xrightarrow[n\rightarrow\infty]{v} W(\x^{-\a}),
\end{equation}

where $\mathbf{x}=\big(x_1,\dots,x_l\big)$, $W(\x)$ is a centred Gaussian random field with covariance given by: $$\forall (\x,\mathbf{y})\in \R^{l}\times \R^{l},\;\; \E\left[W(\x)W(\mathbf{y})\right] = \nu\big(\max(\x^{-1/\a},\mathbf{y^{-1/\a}})\big).$$
\end{thm}

In order to generalize the result stated above to the situation where the sample fraction $k/n$ of observations involved in the local Hill estimator possibly varies with the location, we exploit a property of the inverse of a regularly varying function: if $1-F$ is regularly varying with index $-\a$, then $F^{-1}(1-1/.)$ is regularly varying with index $1/\a$, (see \citet{resnick2007heavy}).

\begin{corrol}
\label{corGaussField}{\sc (A Functional Central Limit Theorem (II))}
Let $\mathbf{k}=(k_1,\dots k_l)\in\N^{\star l}$ with $k_i = k_i(n) \rightarrow \infty$ and $k_i/n\rightarrow 0$ as $n\rightarrow +\infty$. Suppose that, for all $i\in\{1,\;\ldots,\; l\}$, $c_i=\lim_{n\rightarrow\infty} k_1/k_i$ is well-defined and belongs to $]0,\; +\infty[$.
Set $\mathbf{a}(n/\mathbf{k})=(a^{(1)}(n/k_1),\dots,a^{(l)}(n/k_l))$ and $\mathbf{x}'=\left(x_1,c_2^{1/\a}x_2,\dots,c_l^{1/\a}x_l\right)$. We have
\begin{equation}
\label{tailprocesskdiff}
\sqrt{k_1}\left(\frac{1}{k_1}\sum_{i=1}^{n}\mathds{1}\left(\frac{\mathbf{X}_i}{\mathbf{a}(n/\mathbf{k})}>\x\right)-\frac{n}{k_1}\overline{F}\left(\mathbf{a}\left(\frac{n}{\mathbf{k}}\right)\x\right)\right) \xrightarrow[n\rightarrow\infty]{v} W(\x'^{-a}),
\end{equation}

where $W(\x)$ is a centred Gaussian field with same covariance operator as that involved in Theorem \ref{GaussField}.

\end{corrol}

Refer to the Appendix section for the technical proof. Since the local Hill estimators are functionals of the tail empirical process, a Central Limit Theorem for the random vector formed by concatenating them can be immediately derived from Theorem \ref{GaussField}.

\begin{thm}
\label{HillGaussRand}
For any $l\geq 1$ we have, as $n,\; k\rightarrow +\infty$ such that $k=o(n)$:

\begin{equation}
\label{HillRandTh}
\sqrt{k}\left(H_{k,n}^{(1)}-\int_{X_1(k)}^{\infty}\frac{n}{k}\overline{F}_1\left(x\right)\frac{dx}{x},
\dots, H_{k,n}^{(l)}-\int_{X_l(k)}^{\infty}\frac{n}{k}\overline{F}_l\left(x\right)\frac{dx}{x}\right)\Rightarrow \mathcal{N}\left(0,\Sigma\right),
\end{equation}

where $\Sigma_{i,j}=\int_{1}^{\infty}\int_{1}^{\infty}\nu_{i,j}(x,y)x^{-1}y^{-1}dxdy$ and $\Sigma_{i,i} = 2/\a^2$.
\end{thm}

The following corollary relaxes the assumption that all local Hill estimators involve the same number of observations.

\begin{corrol}
\label{corHillGaussRand}
Equipped with the assumptions and notations of Corollary \ref{corGaussField}, for any $l\geq 1$ we have

\begin{equation}
\label{HillRandThkdiff}
\sqrt{k_1}\left(H_{k_1,n}^{(1)}-\int_{X_1(k_1)}^{\infty}\frac{n}{k_1}\overline{F}_1\left(x\right)\frac{dx}{x}, \dots, H_{k_l,n}^{(l)}-\int_{X_l(k_l)}^{\infty}\frac{n}{k_l}\overline{F}_l\left(x\right)\frac{dx}{x}\right) \Rightarrow \mathcal{N}\left(0,\Sigma'\right),
\end{equation}

with $\Sigma'_{i,j}=\int_{c_{j}^{1/\a}}^{\infty}\int_{c_{i}^{1/\a}}^{\infty}\nu_{i,j}(x,y)x^{-1}y^{-1}dxdy$, for $1\leq i \ne j\leq l$ and $\Sigma'_{i,i}=2c_i/\a^2.$
\end{corrol}

We now address the issue of removing the random centring. From a practical perspective indeed, in order to recover a pivotal statistic and build (asymptotic) confidence intervals the random centring should be replaced by $1/\a$. The key point is that $\int_{X_i(k_i)}^{\infty}\frac{n}{k_i}\overline{F}_i\left(x\right)\frac{dx}{x}$ can be substituted for  $\int_{a^{(i)}}^{\infty}\frac{n}{k_i}\overline{F}_i\left(x\right)\frac{dx}{x}$, along with (the second order) Condition \eqref{uni2orderint}. This condition is used when trying to establish a Central Limit Theorem in the univariate setup (see \citet{resnick2007heavy}):
\begin{equation}
\label{uni2orderint}
\forall\,i\in \{1,\ldots,l\},\: \lim_{n\rightarrow\infty}\sqrt{k}\int_{1}^{\infty}\frac{n}{k}\overline{F}_{i}\left(a^{(i)}\left(\frac{n}{k}\right)x\right)-x^{-\alpha}\frac{dx}{x}=0.
\end{equation}
This immediately implies that the random vector
\begin{multline}
\label{hillranddet2}
\sqrt{k}\left(H_{k,n}^{(1)}-\frac{1}{\alpha},\ldots, H_{k,n}^{(l)}-\frac{1}{\alpha}\right)+ \\
\sqrt{k}\left(\int_{a^{(1)}(n/k)}^{X_{1}(k)}\frac{n}{k}\overline{F}_1\left(x\right)\frac{dx}{x},\ldots,\int_{a^{(l)}(n/k)}^{X_{l}(k)}\frac{n}{k}\overline{F}_{l}\left(x\right)\frac{dx}{x}\right)
\end{multline}
converges in distribution to $\mathcal{N}(0,\Sigma')$.
As shown by expression \eqref{hillranddet2}, recentering by $(1/\alpha,\; \ldots,\; 1/\alpha)$ requires to incorporate a term due to the possible correlation between the random centring and the local Hill estimators into the asymptotic covariance matrix. Indeed, from Eq. \eqref{hillranddet2}, we straightforwardly get that
\begin{equation}
\label{hilldet}
\sqrt{k}\left(H_{k,n}^{(1)}-\frac{1}{\alpha},\dots,H_{k,n}^{(l)}-\frac{1}{\alpha}\right)\Rightarrow \mathcal{N}\left(0,\Omega\right),
\end{equation}
as $n$ and $k=o(n)$ both tend to infinity, provided that, for $1\leq i\neq j \leq l$, the expectation of the quantity

\begin{multline}
\label{horrorExpectation}
k \int_{a^{(i)}(n/k)}^{X_i(k)}\frac{n}{k}\overline{F}_i\left(x\right)\frac{dx}{x}\int_{a^{(j)}(n/k)}^{X_j(k)}\frac{n}{k}\overline{F}_j\left(x\right)\frac{dx}{x}\\
+{} k\left(H_{k,n}^{(j)}-\frac{1}{\alpha}\right)\int_{a^{(i)}(n/k)}^{X_i(k)}\frac{n}{k}\overline{F}_i\left(x\right)\frac{dx}{x}+k\left(H_{k,n}^{(i)}-\frac{1}{\alpha}\right)\int_{a^{(j)}(n/k)}^{X_j(k)}\frac{n}{k}\overline{F}_j\left(x\right)\frac{dx}{x}
\end{multline}
converges, the limit being then equal to $\int_{1}^{\infty}\int_{1}^{\infty}\nu_{i,j}(x,y)/(xy)\; dxdy - \Omega_{i,j}$, while $\Omega_{i,i} =  1/\alpha^2$ for all $i \in\{1,\; \ldots,\; l\}$. A tractable expression for the expectation of the quantity in Eq.\eqref{horrorExpectation} (and then for $\Omega$) can be derived from the Bahadur-Kiefer representation of high order quantiles, (see \citet{Csorgo1978}), under the additional hypothesis \eqref{multi2order} which can be viewed as a multivariate counterpart of Condition \eqref{uni2orderint}:

\begin{multline}
\label{multi2order}
\text{For } 1\leq i\ne j\leq l,\:\sup_{x,y>1}\left\vert\frac{n}{k}\overline{F}_{i,j}\left( a^{(i)}\left(\frac{n}{k}\right) x,a^{(j)}\left(\frac{n}{k}\right) y\right)-\nu_{i,j}(x,y) \right\vert\\
=o\left(\frac{1}{\log k}\right)\; \text{as }n,\; k \rightarrow+\infty.
\end{multline}

This condition permits to establish the next theorem, which provides the form of the asymptotic covariance of the r.v. obtained by concatenating the local Hill estimators, when all are recentered by $1/\a$. Corollary \ref{corHillGauss} offers a generalization to the situation where the number of extremal observations involved in the local tail index estimation possibly depends on the location.

\begin{thm}
\label{HillGaussDet}

Suppose that Condition \eqref{uni2orderint} and Condition \eqref{multi2order} hold true, together with the von Mises conditions:
\begin{equation}
\label{vonmise}
\lim_{s\rightarrow\infty}\a(s):=\frac{sF_i'(s)}{1-F_i(s)}= \a,\qquad  \forall i\in\{1,\; \ldots,\; l\}.
\end{equation}
Then we have the convergence in distribution
\begin{equation}
\label{HillDetTh}
\sqrt{k}\left(H^{(1)}_{k,n}-\frac{1}{\alpha},\dots,H_{k,n}^{(l)}-\frac{1}{\alpha}\right)\Rightarrow \mathcal{N}\left(0,\Omega\right),
\end{equation}
\[\text{where }\quad\Omega_{i,j}=\left\{
   \begin{array}{cl}
		\frac{\nu_{i,j}(1,1)}{\a^2} & \mbox{if } 1\leq i\ne j\leq l.\\
		& \\
		\frac{1}{\alpha^2} & \mbox{otherwise}.\\
	\end{array}
   \right.
\]

\end{thm}

\begin{corrol}
\label{corHillGauss}
Suppose that the assumptions of Corollary \ref{corGaussField} are fulfilled and that, for any integer $l\geq 1$ and any $\mathbf{s}=(s_1,\dots,s_l)\in \S^{l}$, conditions  \eqref{uni2orderint}, \eqref{multi2order} and \eqref{vonmise} hold. Then, we have

\begin{equation}
\label{HillDetThkdiff}
\sqrt{k_1}\left(H_{k_1,n}^{(1)}-\frac{1}{\a},
\dots, H_{k_l,n}^{(l)}-\frac{1}{\a}\right)\Rightarrow \mathcal{N}\left(0,\Gamma\right),
\end{equation}

where, for $1\leq i\ne j\leq l$,
\begin{equation*}
\Gamma_{i,j}= \frac{\nu_{i,j}(c_i^{1/\a},c_j^{1/\a})}{\a^2}.
\end{equation*}

\end{corrol}

In order to prepare for the aggregation procedure, we state a central limit theorem for a convex sum of the marginal Hill estimators. 

\begin{thm}
For a given $\l\in\R_{+}^{l}$, such that $\sum_{i=1}^{l}\l_i=1$, we set $H_{\mathbf{k},n}(\l)=\sum_{i=1}^{l}\l_i H_{k_1,n}^{(l)}$. Under the assumptions of Corollary \ref{corHillGauss}, we have

\begin{equation}
\label{cltAggreg}
\sqrt{k_1}\left(H_{\mathbf{k},n}(\l) - \frac{1}{\a}\right) \Rightarrow \mathcal{N}\left(0 , {}^t\l \Gamma\l\right)
\end{equation}

\end{thm}

This result follows directly from Corollary \ref{corHillGauss}. Before showing how this result apply to the aggregation of the local tail index estimators, we exhibit a distribution fulfilling the conditions involved in the previous analysis.

\begin{example}
The $l$-dimensional Gumbel copula $C_{\beta}$ with dependence coefficient $\beta\geq 1$ is given by 
$$C_{\beta}(u_1,\dots,u_l)=\exp\bigg(-\Big((-\log u_1)^{\beta}+\dots+(-\log u_l)^{\beta}\Big)^{1/\beta}\bigg).$$

Let $X_{1},\dots,X_{l}$ be heavy-tailed r.v.'s defined on the same probability space with tail index $\a$, survival functions $\overline{F}_i=1-F_i,\: i =1\dots l$ and with joint distribution $F = C_{\beta}\Big(F_1,\dots,F_l\Big)$. In this case, we have:
$$\forall\, 1\leq i\ne j \leq l,\quad \nu_{i,j}(x,y) = x^{-\a} + y^{-\a} - \left(x^{-\beta\a} + y^{-\beta\a}\right)^{1/\beta}.$$
In addition, Condition \eqref{multi2order} is satisfied if, as $n,\; k \rightarrow\infty$,
\begin{equation}
\label{explmulti}
\text{for } \,1\leq i \leq l,\quad \sup_{x>1}\left\vert \overline{F}_i\left(a^{(i)}\left(\frac{n}{k}\right) x\right) - x^{-\a} \right\vert   = o \left(\frac{1}{\log k}\right).\\
\end{equation}

For instance, if $F_i$ is the Generalized Pareto distribution (GPD) for some $i$, Condition \eqref{explmulti} is satisfied, since in this case $\sup_{x>1}\left\vert \overline{F}_i\left(a_k x\right) - x^{-\a} \right\vert = O\left(\left(k/n\right)^{1/\a}\right)$.
The proof is given in the Appendix \ref{example} therein.

\end{example}

\subsection{Application to AMSE minimization.}
\label{convexaggreg}

Based on the asymptotic results of the previous section, we now consider the problem of building an estimator of the form of a convex sum of the local Hill estimators $H_{k_1,n}^{(1)},\dots,H_{k_l,n}^{(l)}$, namely $H_{\mathbf{k},n}(\l)$, with minimum asymptotic variance. Precisely, the asymptotic mean square error (AMSE in abbreviated form) is defined as $AMSE(\l)= k_1\E\left[\left(H_{\mathbf{k},n}(\l) - \frac{1}{\a}\right)^2\right]$, for $k\in\{k_1,\dots,k_l\}$. Hence, the goal is to find a solution $\l^{\star}=(\l_1^{\star},\dots,\l_l^{\star})$ of the minimization problem
\begin{equation}\label{eq:min_pb}
\min_{\l=(\l_i)_{1\leq i \leq l}\in [0,1]^l} AMSE(\l) \textit{ subject to }\sum_{i=1}^{l} \l_i = 1.
\end{equation}
Observe that under the assumption of Corollary \ref{corHillGauss}, we have 

\begin{equation*}
AMSE(\l)  = k\E\left[\left(\sum_{i=1}^{l}\l_i H_{k_i,n}^{(i)}-\frac{1}{\a}\right)^2\right]
 =\sum_{i=1}^{l}\sum_{j=1}^{l}\l_i\l_j \Gamma_{i,j} = {}^{t}\l\Gamma\l.
\end{equation*} 

The minimization problem \eqref{eq:min_pb} thus boils down to solving the quadratic problem:
$$\underset{C\l\leq d}{\argmin}{}^{t}\l\Gamma\l,$$
where the constraint matrix $C$ and the vector $d$ are given by:

\[C=\left(
   \begin{array}{cccc}
    -1  & -1   & \cdots & -1\\ 
    -1  & 0    & \cdots & 0 \\
    0   & -1   & \cdots & 0 \\
    \vdots    &  & \ddots & \vdots \\ 
    0   & 0  & \cdots & -1 
	\end{array}
   \right)\text{ and } d=\left(
   \begin{array}{c}
    -1  \\ 
    0  \\
    \vdots \\ 
    0 \\
	\end{array}\right).
\]

A variety of procedures can be readily used to solve this quadratic problem, including Uzawa's algorithm for instance, see \citet{Glowinski1990}.

Now, going back to statistical estimation, suppose that a consistent estimator $\widehat{\a}$ of $\a$ is at our disposal. As the matrix $\Gamma$ is unknown it needs to be estimated. We recall that we just need to estimate the $\nu_{i,j}(c_i^{1/\a},c_j^{1/\a}),\: 1\leq i \ne j\leq n $ and not the $l$-dimensional measure $\nu$. Hence, the dimension is not a significant issue here. In practice, we define 
$$\widehat{\nu}_{i,j}(x,y)  = \frac{1}{k}\sum_{m=1}^{n}\mathds{1}\left(\frac{X_m^{(i)}}{X_i(k)}>x,\frac{X_m^{(j)}}{X_j(k)}>y\right)\text{ and }
\widehat{\Gamma}_{i,j}   =  \frac{\widehat{\nu}_{i,j}(c_i^{1/\widehat{\a}},c_j^{1/\widehat{\a}})}{\widehat{\a}^2}.$$
Then, we compute $$\widehat{\l}^{opt}  = \underset{C\l\leq d}{\argmin}{}^{t}\l\widehat{\Gamma}\l.$$ The quantity $H_{\mathbf{k},n}(\widehat{\l}^{opt})$ is then referred to as Best Empirical AggRegation (BEAR) estimator. The performance of the BEAR estimator is investigated from an empirical perspective by means of simulation experiments in the next section.

\section{Simulations}
\label{simulations}
For illustrative purposes, we computed the BEAR estimator for simulated random fields observed on regular grids of dimension $2\times 2$, $3\times 3$ and $4\times 4$.  The distributions of the margins were chosen among the Student-$t$ distribution with degree of freedom $\a$, the Generalized Pareto Distribution (GPD) with shape parameter $1/\a$ the Fr\'echet,  log Gamma, inverse Gamma distribution with shape parameter $\a$, the Burr distribution with shape parameters $\alpha$ and 1. These distributions share the same tail index $\alpha$. The main barrier in practice is the choice of the optimal sample fraction $k$ used to compute the marginal Hill estimators. This choice had to be automated. We implemented the procedures introduced in \citet{Beirlant96,Beirlant96b,Danielsson2001}.  Unfortunately, they lead to inaccurate choices for small sample sizes (except for the $t$-distribution), significantly overestimating the optimal $k$ for the GPD and the Fr\'echet distributions. The corresponding results are not documented. However, it is possible to determine the theoretical optimal value $k_{opt}$ in terms of AMSE. Hence, for each simulated sample we decided to choose at random the optimal $k$ in the interval $\left[\max(30,0.75k_{opt}),\min(n/3,1.25k_{opt})\right]$ where $n$ is the sample size. The largest admissible value is bounded by $n/3$ because we considered that including a larger fraction of the sample would lead to a highly biased estimate. Similarly, we bounded the smallest admissible value in order not to obtain estimates with a too large variance. In addition the interval $\left[0.75k_{opt},1.25k_{opt}\right]$ seemed to be a reasonable choice to account for the error in the selection of the optimal sample fraction.

We tried two different choices of copulas to describe the asymptotic dependence structure of the field, namely the Gumbel copula (\citet{Gudendorf2010}) and the $t$-copula (\citet{Frahm03}). Both yielded very similar results and only the results for the Gumbel copula with dependence parameter $\beta=3$ ($\beta=1$ for exact independence and $\beta\rightarrow\infty$ for exact dependence) are displayed. 

The results for different values are presented in Table \ref{result} where $n$ is the sample size. The AMSE of the BEAR estimator is compared to the AMSE of the Average estimator (Ave.) which is equal to $\left(\overline{k}\right)^{-1}\sum_{i=1}^{l}k_iH_{k_i,n}^{(i)}$ where  $\overline{k}=\sum_{i=1}^{l}k_i$. The \textit{(Impr.)} column indicates the relative improvement in AMSE provided by the BEAR estimator, with respect to the Average estimator.

\begin{table}\centering
\begin{tabular}{ccclcclccl}
\toprule
& & & $\a=1 $ &  & & $\a=2 $ &  & & $\a=5$ \\
 \cmidrule{3-4}  \cmidrule{6-7}  \cmidrule{9-10} dim. & $n$ & BEAR & Ave.(Impr.) && BEAR & Ave.(Impr.) && BEAR & Ave.(Impr.) \\
\midrule\\
  &   1000 & 0.89 & 0.95 (-6\%) & & 0.43 & 0.46 (-7\%) & & 0.11 & 0.12 (-5\%)\\
  &   2500 & 0.97 & 1.08 (-10\%) & & 0.23 & 0.26 (-12\%) & & 0.03 & 0.04 (-10\%)\\
 $ 2\times 2 $ &  5000 & 0.82 & 1.10 (-25\%) & & 0.28 & 0.36 (-23\%) & & 0.04 & 0.05 (-21\%)\\
  &   10000 & 0.80 & 1.12 (-34\%) & & 0.23 & 0.38 (-31\%) & & 0.05 & 0.03 (-30\%)\\
  &   25000 & 0.59 & 1.20 (-51\%) & & 0.18 & 0.34 (-48\%) & & 0.04 & 0.07 (-49\%)\\
\\
  &   1000 & 0.94 & 0.96 (-2\%) & & 0.33 & 0.34 (-4\%) & & 0.10 & 0.10 (5\%)\\
  &   2500 & 1.00 & 1.10 (-9\%) & & 0.25 & 0.29 (-13\%) & & 0.04 & 0.05 (1\%)\\
 $ 3\times 3 $ &  5000 & 0.77 & 0.99 (-22\%) & & 0.26 & 0.32 (-19\%) & & 0.04 & 0.05 (-21\%)\\
  &   10000 & 0.78 & 1.17 (-39\%) & & 0.16 & 0.29 (-37\%) & & 0.04 & 0.05 (-34\%)\\
  &   25000 & 0.68 & 1.28 (-47\%) & & 0.12 & 0.26 (-53\%) & & 0.04 & 0.06 (-41\%)\\
\\
  &   1000 & 0.78 & 0.76 (3\%) & & 0.35 & 0.33 (2\%) & & 0.06 & 0.06 (6\%)\\
  &   2500 & 1.06 & 1.05 (-1\%) & & 0.45 & 0.46 (-3\%) & & 0.07 & 0.07 (3\%)\\
 $ 4\times 4 $ &  5000 & 0.85 & 1.02 (-17\%) & & 0.22 & 0.26 (-16\%) & & 0.04 & 0.05 (-12\%)\\
  &   10000 & 1.00 & 0.86 (-42\%) & & 0.18 & 0.30 (-38\%) & & 0.04 & 0.04 (-29\%)\\
  &   25000 & 0.76 & 1.49 (-49\%) & & 0.14 & 0.29 (-53\%) & & 0.03 & 0.06 (-34\%)\\
\\
\bottomrule
\end{tabular}
\caption{\label{result} Simulation results. AMSE comparison.}
\end{table}

The results show a very good behaviour of the  BEAR estimator, even for relatively small sample sizes $n$. We point out that when $n$ is too small the exponent measure is estimated with less accuracy and the error might contaminate the covariance matrix $\widehat{\Gamma}$ when the dimension is too large.  However, when $n>1000$ the gain in AMSE is significant and the BEAR estimator is much more accurate than the Average estimator and is at least as accurate for $n=1000$. Simulations are presented for sample sizes $n\geq 1000$ that are not restrictive in many contexts such as finance, insurance, reliability based risk assessment with industrial contexts, not to mention the rise of the so-called big data. Hence the BEAR estimator could be used in a wide range of domain.

Note that depending on second order conditions on the marginal distributions, the optimal $k$ may be very small and as a consequence, $\Gamma$ would be estimated with a large variance. This is what explains some poor results when $\a=5$, when the dimension increases (the theoretical optimal $k$ for $n=1000$ for the GPD, inverse Gamma and Burr distributions are smaller than 15). In that case, it might be useful to have alternative choices for the tail index estimator with a larger optimal $k$ and a smaller AMSE. Some ideas to handle this issue are presented in Appendix section \ref{extension}.

\section{Example : sloshing data tail index inference}\label{sloshing}
In the liquefied natural gas (LNG) shipping industry, \textit{sloshing} refers to an hydrodynamic phenomenon which arises when the cargo is set in motion, \citet{PDS}. Following incidents experienced by the ships Larbi Ben M'Hidi and more recently by Catalunya Spirit, these being two LNG carriers faced with severe sloshing phenomena, rigorous risk assessments has become a strong requirement for designers, certification organizations (seaworthiness) and ship owners. In addition, sloshing has also been a topic of interest in other industries (for instance, see \citet{NASA} for a contribution in the field of aerospace engineering).
Gaztransport  $\&$  Technigaz (GTT) is a French company which designs the most widely used cargo containment system (CCS) for conveying LNG, namely the membrane containment system. The technology developed by GTT uses the hull structure of the vessel itself: the tanks are effectively part of the ship. The gas in the cargo is liquefied and kept at a very low temperature $(-163^{\circ}\mathrm{C})$ and atmospheric pressure, thanks to a thermal insulation system which prevents the LNG from evaporating. Although this technology is highly reliable, it can be susceptible to sloshing: waves of LNG apply very high pressures (over 20 bar) on the tank walls on impact and may possibly damage the CCS. Due to its high complexity, the sloshing phenomenon is modelled as a random process. The phenomenon is being studied by GTT experimentally on instrumented small-scale replica tanks ($1:40$ scale), instrumented with pressure sensors arrays. The tanks are shaken by a jack system to reproduce the motion of the ship and induce the occurrence of sloshing, with the associated high pressures being recorded by the sensors.

A pressure measurement is considered as a sloshing impact if the maximal pressure recorded is above 0.05 bar. As soon as a sensor records a pressure above 0.05 bar (we call it an event), the pressures measured by all the other sensors of the array are also recorded at a frequency of 200kHz. The recording stops when the pressures measured by all the sensors are back to zero. For each event and for each sensor, we have a collection of measures. For each sensor  and each event, we only keep the highest pressure.

GTT provided us with a low filling configuration data set: the tanks are nearly empty (the level of LNG in the tank is 10\% of the height of the tank so that only the lower parts of the tank are instrumented with sensors). We consider the observations of a sensors array represented in Fig.\ref{sensorsMatrix} (together with the marginal tail index estimates). This is a $3 \times 3$ sensors array. 48,497 events were recorded by the 9 sensors.

It is the assumption of GTT that the tail index is the same for the observations of all the sensors, even though the field is not supposed to be stationary. This totally fits in the framework of this paper and we use our methodology to estimate the tail index $\a$.  
\medskip

\noindent\textbf{First step : Marginal estimation of $\a$.}
At each location $s_i,\quad i=1\dots 9$, we determine graphically $k_i$ the optimal number of extremes to be used and compute the Hill estimator $H_{k_i,n}^{(i)}$ (the estimtions are displayed in Fig.\ref{sensorsMatrix}). These marginal estimations do not rule out the assumption of equality of the tail indexes. The estimation $\widehat{\a}$ of $\a$ used for the aggregation procedure is the Average estimator defined in the previous section. We found $\widehat{\a}=3.5$, with an estimated AMSE of  $1.1$.

\medskip

\noindent\textbf{Second step : Aggregation.}
We used the methodology described in section \ref{convexaggreg} to compute the BEAR estimator. We found $\widehat{\a}^{opt}=3.6$ with an estimated AMSE of $0.7$. 
\medskip

\begin{figure}[!htb]
\centering
\begin{tikzpicture}[scale=0.9]
	\newcommand{\sensorB}[1]{\draw[fill=gray!40] (#1) circle (0.75);}

	\foreach \c in {1,4,7} {
		\foreach \d in {-1,2,5}{
			 		\sensorB{\c,\d} };}
	\draw (1,-0.6) node{\small $\mathbf{3.6}$};
	\draw (1,-1) node{\footnotesize $(3.2-4.0)$};
	
	\draw (1,2.4) node{\small $\mathbf{3.5}$};
	\draw (1,2) node{\footnotesize $(3.0-4.0)$};
	
	\draw (1,5.4) node{\small $\mathbf{3.7}$};
	\draw (1,5) node{\footnotesize $(3.2-4.2)$};
	
	\draw (4,-0.6) node{\small $\mathbf{3.3}$};
	\draw (4,-1) node{\footnotesize $(2.8-3.9)$};
	
	\draw (4,2.4) node{\small $\mathbf{3.4}$};
	\draw (4,2) node{\footnotesize $(3.0-3.8)$};
	
	\draw (4,5.4) node{\small $\mathbf{3.6}$};
	\draw (4,5) node{\footnotesize $(3.1-4.1)$};
	
	\draw (7,-0.6) node{\small $\mathbf{3.5}$};
	\draw (7,-1) node{\footnotesize $(2.9-4.1)$};
	
	\draw (7,2.4) node{\small $\mathbf{3.7}$};
	\draw (7,2) node{\footnotesize $(3.2-4.2)$};
	
	\draw (7,5.4) node{\small $\mathbf{3.4}$};
	\draw (7,5) node{\footnotesize $(3.1-3.7)$};

\end{tikzpicture}
\caption{Diagram of the $3\times 3$ sensors array with estimated tail index and 95\% confidence interval.}
\label{sensorsMatrix}
\end{figure}
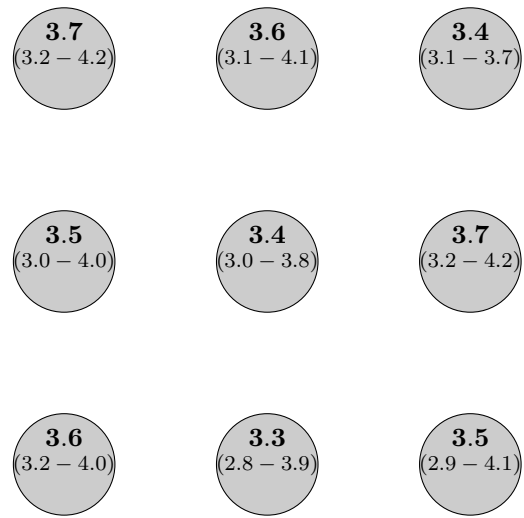

\section{Conclusion}
This paper introduces the so-termed BEAR estimator to estimate its tail index.  Incidentally, the BEAR estimator can also be used in the context of regularly varying random fields or processes (\citet{Hult05}) or in any heavy-tailed multivariate framework, as long as all the margins share the same tail index. Beyond the asymptotic analysis, it was shown to be highly accurate even for small sample sizes and very accurate for typical sample sizes in finance, insurance or in the industry. However, depending on second order conditions of the underlying distributions, when $\a$ increases, some approximations needed to derive asymptotic result may be bad. This can be understood with second. It is the subject of further research to estimate the bias of the BEAR estimator. This study could help deciding which marginal estimator to choose (Hill, Moment or Ratio) in order to minimize the asymptotic mean squared error.

\section*{Appendix - Technical Proofs}
\label{Proofs}

\paragraph{Convention for the remaining of the paper:}
Without loss of generality and for ease of notation, the proofs are given in dimension $l=2$. To lighten, we set $X:=X_1$, $Y:=X_2$, $\nu:=\nu_{1,2}$, $a(n/k)=a^{(1)}(n/k)$ and  $b(n/k)=a^{(2)}(n/k)$.  The survival functions of $X$ and $Y$  are denoted by $\overline{F}_X$ and $\overline{F}_{Y}$ respectively, and the survival function of $(X,Y)$  is denoted by $\overline{F}$ . We observe an $n$-sample $\Big((X_1,Y_1),\dots,(X_n,Y_n)\Big)$ of $(X,Y)$ and for any $i=1\dots k$, we set $U_i = F_X(X_i)$ and $V_i = F_Y(Y_i)$.

\begin{proof}[Proof of Corollary \ref{corGaussField}]
The regular variation property of $a^{(i)}\left(n/k_i\right)$ gives immediately $a^{(i)}\left(n/k_i\right)/a^{(i)}\left(n/k_1\right) \xrightarrow[n\rightarrow\infty]{} c_i^{1/\a}.$

Now, to get the result, we compose the convergence and plug this limit in equation \eqref{tailprocess}, as we do hereinafter in equation \eqref{divhill}.
\end{proof}

\begin{proof}[Proof of Theorem \ref{HillGaussRand}]
In order to obtain Eq. \eqref{HillRandTh}, we will apply a transform on the tail empirical process of Eq.\eqref{tailprocess}.  The tail empirical process will be evaluated at $\x_1=(x,0)$ and $\x_2=(0,y)$. The left-hand term of Eq. \eqref{HillRandTh} will then be obtained by replacing $\ank$ and $\bnk$ by their empirical counterpart. The result will be integrated over $\left(1,\infty\right]$ to obtain the desired convergence. More formally, we first obtain
\begin{multline*}
\sqrt{k}\left(
\frac{1}{k}\sum_{i=1}^{n}\mathds{1}\left(\frac{X_i}{\ank}>x,\frac{Y_i}{\bnk}>0\right)-\frac{n}{k}\overline{F}_X\big(a(n/k)x\big),\right.\\
\left.\frac{1}{k}\sum_{i=1}^{n}\mathds{1}\left(\frac{X_i}{\ank}>0,\frac{Y_i}{\bnk}>y\right)-\frac{n}{k}\overline{F}_{Y}\big(b(n/k)y\big)\right)\Rightarrow \Big(W(\x_1^{-\alpha}),W(\x_2^{-\alpha})\Big),
\end{multline*}
where $W$ is the process defined in Theorem \ref{GaussField}. One also have (\citet[][Eq. (4.17)]{resnick2007heavy})
$$\left(\frac{X(k)}{a(n/k)},\frac{Y(k)}{b(n/k)}\right)\xrightarrow[]{\P} \left(1,1\right).$$
Hence from Proposition 3.1 in \citet{resnick2007heavy} we have
\begin{multline*}
\left(\sqrt{k}\left(
\frac{1}{k}\sum_{i=1}^{n}\mathds{1}\left(\frac{X_i}{\ank}>x,\frac{Y_i}{\bnk}>0\right)-\frac{n}{k}\overline{F}_X\big(a(n/k)x\big),\right.\right.\\
\left.\left.\frac{1}{k}\sum_{i=1}^{n}\mathds{1}\left(\frac{X_i}{\ank}>0,\frac{Y_i}{\bnk}>y\right)-\frac{n}{k}\overline{F}_{Y}\big(b(n/k)y\big)\right),\right.\\
\left.\vphantom{\sum_{i=1}^{n}}\left(\frac{X(k)}{a(n/k)},\frac{Y(k)}{b(n/k)}\right)\right)
\Rightarrow \Bigg(\Big(W\big(\x_1^{-\alpha}\big),W\big(\x_2^{-\alpha}\big)\Big),\big(1,1\big)\Bigg).
\end{multline*}
Now, we apply the composition map $(x(t),p)\mapsto x(tp)$ which gives
\begin{multline}
\label{divhill}
\sqrt{k}\left(
\frac{1}{k}\sum_{i=1}^{n}\mathds{1}\left(\frac{X_i}{X(k)}>x,\frac{Y_i}{Y(k)}>0\right)-\frac{n}{k}\overline{F}_X\big(X(k)x\big),\right.\\
\left.\frac{1}{k}\sum_{i=1}^{n}\mathds{1}\left(\frac{X_i}{X(k)}>0,\frac{Y_i}{Y(k)}>y\right)-\frac{n}{k}\overline{F}_{Y}\big(Y(k)y\big)\right)\Rightarrow \Big(W(\x_1^{-\alpha}),W(\x_2^{-\alpha})\Big),
\end{multline}
Equation \eqref{divhill} yields, again by \citet[][p.297-298]{resnick2007heavy}, as $n,k\rightarrow\infty$:
\begin{multline}
\label{hillint}
\sqrt{k}\left(\int_{1}^{\infty}\frac{1}{k}\sum_{i=1}^{n}\mathds{1}\left(\frac{X_i}{X(k)}>x\right)\frac{dx}{x}-\int_{1}^{\infty}\frac{n}{k}\overline{F}_X\big(X(k)x\big)\frac{dx}{x},\right.\\
\left.\int_{1}^{\infty}\frac{1}{k}\sum_{i=1}^{n}\mathds{1}\left(\frac{Y_i}{Y(k)}>y\right)\frac{dy}{y}-\int_{1}^{\infty}\frac{n}{k}\overline{F}_Y\big(Y(k)y\big)\frac{dy}{y}\right)\\
\Rightarrow \left(\int_{1}^{\infty}W(\x_1^{-\alpha})\frac{dx}{x},\int_{1}^{\infty}W(\x_2^{-\alpha})\frac{dy}{y}\right)
\end{multline}
with 
\begin{eqnarray*}
\mathbb{V}\left[\int_{1}^{\infty}W(\x_1^{-\alpha})\frac{dx}{x}\right] & = & \frac{2}{\alpha^2}.\\
\mathbb{C}\text{ov}\left[\int_{1}^{\infty}W(\x_1^{-\alpha})\frac{dx}{x},\int_{1}^{\infty}W(\x_2^{-\alpha})\frac{dy}{y}\right] & = &\int_{1}^{\infty}\int_{1}^{\infty}\nu(x,y)\frac{dxdy}{xy}:=\sigma^2.
\end{eqnarray*}
Equation \eqref{hillint} is equivalent to 
\begin{equation}
\label{hillrand}
\sqrt{k}\left(H_{k,n}^{X}-\int_{1}^{\infty}\frac{n}{k}\overline{F}_X\big(X(k)x\big)\frac{dx}{x},H_{k,n}^{Y}-\int_{1}^{\infty}\frac{n}{k}\overline{F}_{Y}\big(Y(k)y\big)\frac{dy}{y}\right)\Rightarrow K
\end{equation}

\[\text{where } K\sim\mathcal{N}\left(0,\left(
   \begin{array}{cc}
		\frac{2}{\alpha^2} & \sigma^2\\
		\sigma^2 & \frac{2}{\alpha^2}\\
	\end{array}\right)
   \right).
\]

\end{proof}

\begin{proof}[Proof of Corollary \ref{corHillGaussRand}]
The argument is the same as for Corollary  \ref{corGaussField}.
\end{proof}

\begin{proof}[Proof of Theorem \ref{HillGaussDet}]
In order to prove Theorem \ref{HillGaussDet}, we need to give an analytic expression to the covariance matrix in Eq. \eqref{HillDetTh}. This is the object of the Lemmas \eqref{bias1} and \eqref{bias2}.

\begin{lem}
\label{bias1}
Under Condition \eqref{vonmise}, we have
\begin{equation*}
\lim_{k\rightarrow \infty}\E\left[k\int_{\ank}^{X(k)}\frac{n}{k}\overline{F}_X\left(x\right)\frac{dx}{x}\int_{\bnk}^{Y(k)}\frac{n}{k}\overline{F}_Y\left(x\right)\frac{dx}{x}\right]=\frac{\nu(1,1)}{\a^2}.
\end{equation*}
\end{lem}

\begin{lem}
\label{bias2}
Under Condition \eqref{multi2order}, we have
\begin{equation*}
\lim_{k\rightarrow\infty}\E\left[k\left(H_{k,n}^{X}-\frac{1}{\alpha}\right)\int_{\bnk}^{Y(k)}\frac{n}{k}\overline{F}_Y\left(x\right)\frac{dx}{x}\right]  = \frac{1}{\a}\int_{1}^{\infty}\frac{\nu\left(x,1\right)}{x}dx  - \frac{\nu(1,1)}{\a^2}.
\end{equation*}
\end{lem}

To show the Lemmas \ref{bias1} and \ref{bias2}, we linearise functional of the order statistics $X(1),\dots,X(k),$ $Y(1),\dots,Y(k)$ of $X$ and $Y$ as series of the original observations $X_1,\dots,X_n,$ $Y_1,\dots,Y_n$. This is done, using Taylor series and the Bahadur-Kiefer representation  of the order statistics (\citet{Bahadur66}), in Lemmas \ref{linearisation} and \ref{orders}. The Bahadur-Kiefer representation involves a remainder term (see \citet{Kiefer67})  that needs to be controlled. This is the object of Lemmas \ref{logorder}, \ref{limrap} and \ref{riemansum}. The Lemmas \ref{bias1} and \ref{bias2} put together all the results of the aforementioned Lemmas. Lemma \ref{simplification} simplifies the expressions given in Lemmas \ref{bias1} and \ref{bias2}.

\end{proof}

\begin{lem}[\sc{Bahadur-Kiefer representations}]
\label{linearisation}

We set $p_i=\frac{n-i+1}{n}$ and $\overline{p}_i=1-p_i,\quad i=1\dots k$, we have the almost sure equalities
\begin{align}
\label{BahadurKiefer}
X(i) & = a(n/i) -\frac{1}{n}\sum_{j=1}^{n}\frac{\mathds{1}\left(U_j\leq p_i\right)-p_i}{f_X\Big(a\big(n/i\big)\Big)} + T_n(p_i)\\
\label{approxint}
\int_{a(n/k)}^{X(k)}\frac{n}{k}\overline{F}_Y\left(x\right)\frac{dx}{x} & = -\frac{1}{n}\sum_{j=1}^{n}\frac{\mathds{1}\left(U_j\leq p_k\right)-p_k}{\ank f_X\Big(a\big(n/k\big)\Big)} + \frac{T_n(p_k)}{\ank}.\\
\label{BKlog}
\log X(i) & = \log a\big(n/i\big) -\frac{1}{n}\sum_{j=1}^{n}\frac{\mathds{1}\left(U_j\leq p_i\right)-p_i}{a\big(n/i\big)f_X\Big(a\big(n/i\big)\Big)} + O\left(\frac{T_n(p_i)}{a\big(n/i\big)}\right).
\end{align}
where  $T_n$ is a remainder terms.
\end{lem}

\begin{proof}[Proof of Lemma \ref{linearisation}]

Eq. \eqref{BahadurKiefer} is just the Bahadur Kiefer representation of $X(i)$. For Eq. \eqref{approxint} we have almost surely

\begin{align*}
\int_{a(n/k)}^{X(k)}\frac{n}{k}\overline{F}_X\left(x\right)\frac{dx}{x} & = \int_{1}^{\frac{X(k)}{a(n/k)}}\frac{n}{k}\overline{F}_X\left(\an2 x\right)\frac{dx}{x}\\
& = \frac{1}{\a}\left(1-\left(\frac{X(k)}{a(n/k)}\right)^{-\a}\right)+O\left(\frac{1}{\sqrt{k}}\right)\\
& = \frac{X(k)-\ank}{\ank} + O\left( \left(\frac{X(k)-\ank}{\ank}\right)^2\right).\\
\end{align*}

Eq. \eqref{BKlog} follows from a Taylor expansion of the logarithm function.

\end{proof}

\begin{lem}[\sc{Control of the remainder terms}]
\label{orders}

Assuming Conditions (3.2) and (3.4) in \citet{Csorgo1978}, we have almost surely:
\begin{align}
\label{majsuprest}
\sup_{0\leq y  \leq 1}\vert T_n(y)\vert & = O\left(n^{-3/4}\left(\log\log n\right)^{-1/4}\left(\log n\right)^{-1/2}\right).\\
\label{majsupint}
\int_{b(n/k)}^{Y(k)}\frac{n}{k}\overline{F}_Y\left(x\right)\frac{dx}{x} & = O\left(\frac{1}{\ank}n^{-3/4}\left(\log\log n\right)^{-1/4}\left(\log n\right)^{-1/2}\right).
\end{align}

\end{lem}

\begin{proof}[Proof of Lemma \ref{orders}]
Eq.\eqref{majsuprest} follows directly from Th.E in \citet{Csorgo1978}.\\

In order to prove Eq. \eqref{majsupint}, we start recalling some fact about uniform approximation of the generalized quantile process. We set 
\begin{align}
\label{genquant}
\rho^Y_n(p_i) & = \sqrt{n}\Big( Y(i) - b(n/i)\Big)f_Y\Big(b\big(n/i\big)\Big)\\
\label{unifquant}
u_n^Y(p_i) & = \sqrt{n}\left( V(i) - p_i \right)
\end{align}

It is known from \citet{Csorgo1978} that under specific conditions satisfied by regularly varying survival functions, we have

\begin{align}
\label{approxdiff}
\sup_{0\leq y\leq 1}\left\vert\rho^Y_n(y)-u^Y_n(y)\right\vert & = O\left(n^{-1/2}\log\log n \right) a.s.\\
\label{approxunif}
\sup_{0\leq y\leq 1}\left\vert u^Y_n(y)\right\vert & = O\left(n^{-1/4}\left(\log\log n \right)^{-1/4}\left(\log n\right)^{-1/2}\right) a.s.,
\end{align}
\textit{cf} \citet[see][cond. (i) to (iv) p.18]{Csorgo2004}.
We deduce for Eq. \eqref{approxdiff} and \eqref{approxunif} that

\begin{equation}
\label{approxgen}
\sup_{0\leq y\leq 1}\left\vert\rho^Y_n(y)\right\vert = O\left(n^{-1/4}\left(\log\log n \right)^{-1/4}\left(\log n\right)^{-1/2}\right) a.s.
\end{equation}

Now, for Eq. \eqref{majsupint}, notice that 

\begin{equation*}
\int_{\bnk}^{Y(k)}\frac{n}{k}\overline{F}_Y\left(x\right)\frac{dx}{x}  = \frac{1}{\a}\frac{\rho_n(p_k)}{\sqrt{n}\bnk} + O\left( \left(\frac{\rho_n(p_k)}{\sqrt{n}\bnk}\right)^2\right)\quad a.s.
\end{equation*}

and conclude by means of Eq. \eqref{approxgen}
\end{proof}

\begin{lem}[\sc{Covariance computation (I)}]
\label{logorder}
$$\E\left[\log X(i) \int_{b(n/k)}^{Y(k)}\frac{n}{k}\overline{F}_X\left(x\right)\frac{dx}{x}\right] = M_n(i) + R_{n,1}(k)+R_{n,2}(k),$$
where
\begin{align*}
M_{n}(i) & =\frac{\P\Big(X>a\big(n/i\big),Y>\bnk\Big)-\overline{p}_i\overline{p}_k}{n a(n/i) f_X\Big(a\big(n/i\big)\Big) \bnk f_Y\Big(b\big(n/k\big)\Big)},\\
R_{n,1}(k) & = O\left(\frac{n^{-3/2}\left(\log\log n\right)^{-1/2}\left(\log n\right)^{-1}}{\ank\bnk}\right),\\
R_{n,2}(k) & = O\left(c\cdot\bnk^{-1}n^{-3/4}\left(\log\log n\right)^{-1/4}\left(\log n\right)^{-1/2}\right).
\end{align*}

\end{lem}

\begin{proof}[Proof of Lemma \ref{logorder}]

One may write
\begin{multline*}
\E\left[\log X(i)\int_{b(n/k)}^{Y(k)}\frac{n}{k}\overline{F}_X\left(x\right)\frac{dx}{x}\right] =\\
\E\left[\left(\log a(n/i)-\frac{1}{n}\sum_{j=1}^{n}\frac{\mathds{1}\left(U_j\leq p_i\right)-p_i}{a(n/i)f_X\Big(a\big(n/i\big)\Big)}\right)\left(-\frac{1}{n}\sum_{j=1}^{n}\frac{\mathds{1}\left(V_j\leq p_k\right)-p_k}{\bnk f_Y\Big(b\big(n/k\big)\Big)}\right)\right] \\
+\mathbb{E}\left[O\left(\frac{T_n(p_i)}{a(n/i)}\int_{b(n/k)}^{Y(k)}\frac{n}{k}\overline{F}_Y\left(x\right)\frac{dx}{x}\right)\right]+\mathbb{E}\left[ \log X(i) \frac{T_n(p_k)}{\bnk}\right]\\
\shoveleft{=M_n(i) + R_{n,1}(k) + R_{n,2}(k)}.\\
\end{multline*}
We have
\begin{align*}
M_n(i)  & = \E\left[\left(\frac{1}{n^2}\sum_{j=1}^{n}\frac{\Big(\mathds{1}\big(U_j\leq p_i\big)-p_i\Big)\Big(\mathds{1}\big(V_j\leq p_k\big)-p_k\Big)}{n a(n/i) f_X\Big(a\big(n/i\big)\Big) \bnk f_Y\Big(b\big(n/k\big)\Big)}\right)\right]\\
        & = \frac{\P\Big(X>a\big(n/i\big),Y>\bnk\Big)-\overline{p}_i\overline{p}_k}{n a(n/i) f_X\Big(a\big(n/i\big)\Big) \bnk f_Y\Big(b\big(n/k\big)\Big)}.
\end{align*}

By Lemma \ref{orders}, we have 
\begin{equation*}
R_{n,1}(k)  = O\left(\frac{n^{-3/2}\left(\log\log n\right)^{-1/2}\left(\log n\right)^{-1}}{\ank\bnk}\right).
\end{equation*}

In addition, by virtue of Cauchy-Schwarz inequality, we have
\begin{align*}
R_{n,2}(i) & = \mathbb{E}\left[ \log X(i) \frac{T_n(p_k)}{\bnk}\right]\\
           & \leq \bnk^{-1}\sqrt{\mathbb{E}\left[ \log^2 X(i) \right]\mathbb{E}\left[T_n(p_k)^2\right]}.\\
\end{align*}

For any positive value of $\a$, $\mathbb{E}\left[ \log^2 X(i) \right]<+\infty$. Hence, we have

\begin{equation*}
R_{n,2}(i) = O\left(\bnk^{-1}n^{-3/4}\left(\log\log n\right)^{-1/4}\left(\log n\right)^{-1/2}\right). \\
\end{equation*}

\end{proof}

\begin{lem}[\sc{Covariance computation (II)}]
\label{limrap}
The following asymptotic relationships hold
\begin{equation*}
\forall i= 1\dots k, \qquad M_n(i)  \sim \frac{1}{\a^2}\frac{n}{ik}\P\Big(X>a\big(n/i\big),Y>\bnk\Big).
\end{equation*}
In the particular case $i=k,$ we have
$$M_n(k) \sim \frac{\nu(1,1)}{k\a^2}.$$

\end{lem}

\begin{proof}[Proof of Lemma \ref{limrap}]

Von Mise's Conditions \eqref{vonmise} give
\begin{equation}
\label{vonmiseequi}
a\big(n/i\big) f_X\Big(a\big(n/i\big)\Big) \bnk f_Y\Big(\bnk\Big) \sim  \a^2\frac{ik}{n^2}.
\end{equation}
With Eq. \eqref{vonmiseequi}, this yields
\begin{equation*}
\frac{\P\Big(X>a\big(n/i\big),Y>\bnk\Big)-\overline{p}_k\overline{p}_k}{n a\big(n/i\big) f_X\Big(a\big(n/i\big)\Big) \bnk f_X\Big(\bnk\Big)} \sim  \frac{1}{\a^2}\frac{n}{ik}\P\Big(X>a\big(n/i\big),Y>\bnk\Big).
\end{equation*}

\end{proof}

\begin{lem}[\sc{Covariance computation (III)}]
\label{riemansum}
The following convergence holds
$$\lim_{n\rightarrow\infty} \sum_{i=1}^{k}\frac{n}{ki}\P\Big(X>a\big(n/i\big),Y>b(n/k)\Big)  = \int_{1}^{\infty}\frac{\nu(x,1)}{x}dx.$$
\end{lem}

\begin{proof}
Write, for any $i=1\dots k$
\begin{align*}
\frac{n}{ki}\P\Big(X>a\big(n/i\big),Y>b(n/k)\Big)  & = \frac{n}{ki}\P\left(X>\an2 \frac{a\left(\frac{n}{i}\right)}{\an2 },Y>b\left(\frac{n}{k}\right)\right)\\
  & = \frac{n}{ki}\overline{F}\left( \an2 \frac{a\left(\frac{n}{i}\right)}{\an2 }
, \bn2 \right).
\end{align*}
We have
\begin{multline*}
\sup_{x>1}\left\vert\sum_{i=1}^{k}\frac{\frac{n}{k}\overline{F}\left( \an2 x, \bn2 \right)}{i}- \frac{\nu\left(x,1\right)}{i} \right\vert \\
\leq \sup_{x>1}\sum_{i=1}^{k}\frac{1}{i} \left\vert\frac{n}{k}\overline{F}\left(\an2 x,\bn2 \right) - \nu\left(x,1\right) \right\vert \\
\leq \sum_{i=1}^{k}\frac{1}{i} \sup_{x>1}\left\vert\frac{n}{k}\overline{F}\left(\an2 x,\bn2 \right) - \nu\left(x,1\right) \right\vert \\
\sim \log k .\sup_{x>1}\left\vert\frac{n}{k}\overline{F}\left(\an2 x,\bn2 \right) - \nu\left(x,1\right) \right\vert\\
\hspace{-5cm}\sim o\left(1\right)\qquad \text{by Condition \eqref{multi2order}}.
\end{multline*}

Then,
\begin{multline*}
\left\vert\sum_{i=1}^{k}\frac{1}{i}\left(\frac{n}{k}\overline{F}\left( \an2 \frac{a\left(\frac{n}{i}\right)}{\an2 }
, \bn2 \right)- \nu\left(\frac{a\left(\frac{n}{i}\right)}{\an2 },1\right)\right) \right\vert \\
 \leq {}\sup_{x>1}\sum_{i=1}^{k}\frac{1}{i} \left\vert\frac{n}{k}\overline{F}\left(\an2 x,\bn2 \right) - \nu\left(x,1\right) \right\vert \\
\end{multline*}
Hence 
$$\left\vert\sum_{i=1}^{k}\frac{1}{i}\frac{n}{k}\overline{F}\left( \an2 \frac{a\left(\frac{n}{i}\right)}{\an2 }
, \bn2 \right)- \sum_{i=1}^{k}\frac{1}{i}\nu\left(\frac{a\left(\frac{n}{i}\right)}{\an2 },1\right)\right\vert\xrightarrow[n\rightarrow\infty]{} 0.$$

In addition, using Potter's Bound, for any $\e>0$ if $n$ is large enough we have

\begin{equation*}
(1-\e)\left(\frac{k}{i}\right)^{1/\a-\e} \leq  \frac{a\left(\frac{n}{i}\right)}{\an2 } \leq (1+\e)\left(\frac{k}{i}\right)^{1/\a+\e}
\end{equation*}

\begin{equation*}
\nu\left((1-\e)\left(\frac{k}{i}\right)^{1/\a-\e},1\right) \leq  \nu\left(\frac{a\left(\frac{n}{i}\right)}{\an2 },1\right) \leq  \nu\left( (1+\e)\left(\frac{k}{i}\right)^{1/\a+\e},1\right)
\end{equation*}

\begin{multline}
\label{riemBond}
\frac{1}{k}\sum_{i=1}^{k}\frac{k}{i}\nu\left((1-\e)\left(\frac{k}{i}\right)^{1/\a-\e},1\right) \leq \frac{1}{k}\sum_{i=1}^{k}\frac{k}{i}\nu\left(\frac{a\left(\frac{n}{i}\right)}{\an2 },1\right) \\
\leq  \frac{1}{k}\sum_{i=1}^{k}\frac{k}{i}\nu\left( (1+\e)\left(\frac{k}{i}\right)^{1/\a+\e},1\right). 
\end{multline}

As $n\rightarrow \infty$ and $\e\rightarrow 0$,	the bounds of Eq. \eqref{riemBond} converges towards \\
$\int_{0}^{1}\nu\left(x^{-1/\a},1\right)x^{-1}dx$. We deduce from the above that 

\begin{equation*}
\frac{1}{k}\sum_{i=1}^{k}\frac{k}{i}\nu\left(\frac{a\left(\frac{n}{i}\right)}{\an2 },1\right) \xrightarrow[n\rightarrow\infty]{}\int_{0}^{1}\frac{\nu\left(x^{-1/\a},1\right)}{x}dx=\a\int_{1}^{\infty}\frac{\nu\left(x,1\right)}{x}dx
\end{equation*}

Finally, we obtain the desired convergence 

\begin{equation*}
\sum_{i=1}^{k}\frac{1}{i}\frac{n}{k}\overline{F}\left( \an2 \frac{a\left(\frac{n}{i}\right)}{\an2 }
, \bn2 \right)  \xrightarrow[n\rightarrow\infty]{} \a\int_{1}^{\infty}\frac{\nu\left(x,1\right)}{x}dx.
\end{equation*}
\end{proof}

\begin{proof}[Proof of Lemma \ref{bias1}]
We have 
\begin{multline*}
k\E\left[\int_{a(n/k)}^{X(k)}\frac{n}{k}\overline{F}_X\left(x\right)\frac{dx}{x}\int_{b(n/k)}^{Y(k)}\frac{n}{k}\overline{F}_Y\left(x\right)\frac{dx}{x}\right] =\\
\end{multline*}
\vspace{-1.5cm}
\begin{multline*}
\shoveleft{k\E\left[\left(-\frac{1}{n}\sum_{j=1}^{n}\frac{\mathds{1}\left(U_j\leq p_k\right)-p_k}{\ank f_X\Big(\ank\Big)} + \frac{T_n(p_k)}{\ank}\right)\right.}\\
\left. \times{}\left(-\frac{1}{n}\sum_{j=1}^{n}\frac{\mathds{1}\left(V_j\leq p_k\right)-p_k}{\ank f_Y\Big(\bnk\Big)} + \frac{T_n(p_k)}{\bnk}\right)\right]\\
\end{multline*}
\vspace{-1.5cm}
\begin{multline*}
=k\E\left[\frac{1}{n^2}\sum_{j=1}^{n}\frac{\Big(\mathds{1}\left(U_j\leq p_k\right)-p_k\Big)\Big(\mathds{1}\left(V_j\leq p_k\right)-p_k\Big)}{\ank f_X\Big(\ank\Big) \bnk f_X\Big(\bnk\Big)}\right]\\
+{}\E\left[\frac{T_n(p_k)}{\ank}\int_{b(n/k)}^{Y(k)}\frac{n}{k}\overline{F}_Y\left(x\right)\frac{dx}{x} \right] +\E\left[\frac{T_n(p_k)}{\bnk}\int_{a(n/k)}^{X(k)}\frac{n}{k}\overline{F}_X\left(x\right)\frac{dx}{x} \right]\\
\end{multline*}
\vspace{-1.5cm}
\begin{multline*}
={} k\frac{\P\Big(X>\ank,Y>b(n/k)\Big)-\overline{p}_k\overline{p}_k}{n \bnk f_X\Big(b\big(n/k\big)\Big) \ank f_X\Big(\ank\Big)} \\
+{} O\left(k\frac{n^{-3/2}\left(\log\log n\right)^{-1/2}\left(\log n\right)^{-1}}{\ank\bnk}\right).
\end{multline*}
The result follows from Lemma \ref{limrap}.

\end{proof}

\begin{proof}[Proof of Lemma \ref{bias2}]
Using the notations above, we can write  

\begin{multline*}
\lim_{k\rightarrow\infty}\E\left[k\left(H_{k,n}^{X}-\frac{1}{\alpha}\right)\int_{b(n/k)}^{Y(k)}\frac{n}{k}\overline{F}_Y\left(x\right)\frac{dx}{x}\right] \\
\shoveleft{= \lim_{k\rightarrow\infty}\frac{k}{\a}\E\left[\left(\frac{1}{k}\sum_{i=1}^{k}\log{\frac{X(i)}{X(k)}}-\frac{1}{\a}\right)\int_{b(n/k)}^{Y(k)}\frac{n}{k}\overline{F}_Y\left(x\right)\frac{dx}{x}\right]}\\
\shoveleft{ = \lim_{k\rightarrow\infty}\frac{k}{\a}\E\left[\left(\frac{1}{k}\sum_{i=1}^{k}\log X(i)-\log X(k) - \frac{1}{\a}\right)\int_{b(n/k)}^{Y(k)}\frac{n}{k}\overline{F}_Y\left(x\right)\frac{dx}{x}\right]}\\
\shoveleft{=\lim_{k\rightarrow\infty} \sum_{i=1}^{k}M_n(i)-M_n(k) + \lim_{k\rightarrow\infty} \sum_{i=1}^{k}R_{n,1}(i)+R_{n,2}(i)-R_{n,1}(k)-R_{n,2}(k)}\\
\shoveleft{= \lim_{k\rightarrow\infty} \sum_{i=1}^{k}M_n(i)-M_n(k)\qquad\text{by Lemma \ref{orders}}.}\\
\shoveleft{=\lim_{n\rightarrow\infty}\frac{1}{\a^2} \sum_{i=1}^{k}\frac{n}{ki}\P\Big(X>a\big(n/i\big),Y>b(n/k)\Big)- \frac{\nu(1,1)}{\a^2}\qquad\text{by Lemma \ref{limrap}}.}\\
= \frac{1}{\a}\int_{1}^{\infty}\frac{\nu\left(x,1\right)}{x}dx - \frac{\nu(1,1)}{\alpha^2}\qquad\text{by Lemma \ref{riemansum}}.
\end{multline*}
\end{proof}

\begin{lem}
\label{simplification}
$$\int_{1}^{\infty}\int_{1}^{\infty}\frac{\nu(x,y)}{xy}dxdy - \frac{1}{\a}\int_{1}^{\infty}\frac{\nu\left(x,1\right)}{x}dx - \frac{1}{\a}\int_{1}^{\infty}\frac{\nu\left(1,y\right)}{y}dy.$$
\end{lem}

\begin{proof}[Proof of Lemma \ref{simplification}]
For $1\leq i\ne j\leq l$, we have

\begin{align*}
\frac{1}{\a}\int_{1}^{\infty}\frac{\nu_{i,j}\left(x,1\right)}{x}dx & = \int_{1}^{\infty}\int_{y}^{\infty}\frac{\nu_{i,j}(x,y)}{xy}dxdy\\
\text{and}\quad \frac{1}{\a}\int_{1}^{\infty}\frac{\nu_{i,j}\left(1,y\right)}{y}dy & = \int_{1}^{\infty}\int_{x}^{\infty}\frac{\nu_{i,j}(x,y)}{xy}dxdy.\\
\end{align*}
Now, just notice that 
$$\int_{1}^{\infty}\int_{1}^{\infty}\frac{\nu_{i,j}(x,y)}{xy}dxdy = \int_{1}^{\infty}\int_{y}^{\infty}\frac{\nu_{i,j}(x,y)}{xy}dxdy + \int_{1}^{\infty}\int_{x}^{\infty}\frac{\nu_{i,j}(x,y)}{xy}dxdy.$$
\end{proof}

\section{Example - Technical details}
We first treat the case $l=2$. We set $F_1 = G$, $F_2 = H$, $a^{(1)}\big(n/k\big)=a_k$ and $a^{(2)}\big(n/k\big)=b_k$
\label{example}

We have $\overline{F}\big(a_k x,b_k y\big) =\overline{G}\left(a_k x\right) + \overline{H}\left(b_k y\right) - 1 +  C_{\beta}\Big(G\left(a_k x\right),H\left(b_k y\right)\Big)$, and 
\begin{multline*}
\frac{n}{k}C_{\beta}\Big(G\left(a_k x\right),H\left(b_k y\right)\Big)  - \frac{n}{k} \\
 = \bigg(\Big(\frac{n}{k}\overline{G}\left(a_k x\right)\Big)^{\beta}+\Big(\frac{n}{k}\overline{H}\left(b_k y\right)\Big)^{\beta}\bigg)^{1/\beta} 
\quad {}+ O\bigg(\Big(\frac{n}{k}\overline{G}\left(a_k x\right)\Big)^{\beta}+\Big(\frac{n}{k}\overline{H}\left(b_k y\right)\Big)^{\beta}\bigg)^{1/\beta}\\
  ={} \bigg(\Big(\frac{n}{k}\overline{G}\left(a_k x\right)-x^{-\a}+x^{-\a}\Big)^{\beta}+\Big(\frac{n}{k}\overline{H}\left(b_k y\right)+y^{-\a}-y^{-\a}\Big)^{\beta}\bigg)^{1/\beta} \\
\quad {}+ O\bigg(\Big(\frac{n}{k}\overline{G}\left(a_k x\right)\Big)^{\beta}+\Big(\frac{n}{k}\overline{H}\left(b_k y\right)\Big)^{\beta}\bigg)^{1/\beta}\\
  ={} \bigg(x^{-\beta\a}\Big(1+\frac{\frac{n}{k}\overline{G}\left(a_k x\right)-x^{-\a}}{x^{-\a}}\Big)^{\beta}+\frac{1}{y^{\a\beta}}\Big(1+\frac{\frac{n}{k}\overline{H}\left(b_k y\right)-y^{-\a}}{y^{-\a}}\Big)^{\beta}\bigg)^{1/\beta} \\
    ={} \left(x^{-\beta\a}+y^{-\beta\a}\right)^{1/\beta} + O\left(\frac{\frac{n}{k}\overline{G}\left(a_k x\right)-x^{-\a}}{x^{-\a}} +\frac{\frac{n}{k}\overline{H}\left(b_k y\right)-y^{-\a}}{y^{-\a}} \right) 
\end{multline*} 

This gives
\begin{small}
\begin{equation*}
\sup_{x,y>1}\left\vert \frac{n}{k}\overline{F}\big(a_k x,b_k y\big)-\nu(x,y)\right\vert  = O\left(   \sup_{x>1}\left\vert \overline{G}\left(a_k x\right) - \frac{1}{x^{\a}} \right\vert +\sup_{y>0}\left\vert \overline{H}\left(b_k y\right) - \frac{1}{x^{\a}} \right\vert  \right).\\
\end{equation*}
\end{small}

Now, in the general case, it can easily be shown that for $x_i > 0,\:i=1,\dots,d$, we have
\begin{equation*}
\nu(x_1,\dots,x_l) = \sum_{k=1}^{l}\sum_{1\leq i_1\ne \dots \ne i_k\leq l}(-1)^{k+1}\left(x_{i_1}^{-\a\beta}+\dots+x_{i_k}^{-\a\beta}\right)^{\beta}.
\end{equation*}

It follows that for any $1\leq i\ne j \leq l$, we have:

$$\nu_{i,j}(x,y) = \left(x^{-\beta\a}+y^{-\beta\a}\right)^{1/\beta}.$$

\section{Extensions to alternative tail index estimation methods}
We now give an insight into the way the BEAR procedure can be generalised to alternative local estimators of the tail index in the case when the Hill estimator does not have a good behaviour.

\label{extension}
Given a sample $X_1,\; \ldots,\; X_n$ in of heavy-tailed variables with the same tail index $\a$, for some $\a>0$, and the related order statistics $X(1)>\dots>X(n)$, \citet{Dekkers89} introduced the moment estimator $M_{k,n}$ defined as
\begin{equation}
\label{moment}
M_{k,n}=H_{k,n}+1-\frac{1}{2}\left(1-\frac{H^2_{k,n}}{L_{k,n}}\right)^{-1},
\end{equation}
where \begin{equation}
\label{Lkn}
L_{k,n}=\frac{1}{k}\sum_{i=1}^{k}\log^2\left(\frac{X(i)}{X(k+1)}\right).
\end{equation}
Note that $L_{n,k}$ is an estimator of $2/\a^2$. \citet{Vries1996} also introduced the ration estimator $J_{k,n}$ defined as 
\begin{equation}
\label{Vries}
J_{k,n} = \frac{L_{k,n}}{2H_{k,n}}
\end{equation}

This estimator was used in \citet{Danielsson2001} to derive the optimal number $k$ of upper order statistics through a bootstrap method.

\paragraph{A Central Limit Theorem for $M_{k,n}$ and $J_{k,n}$}
We use the same notations as those introduced at the beginning of section \ref{MainTh} and consider a $l$-dimensional vector $\mathbf{X}$ of regularly varying margins with index $-\a$.  Following step by step the proofs of the main results of section \ref{MainTh}, one may adapt them to obtain a multivariate Central Limit Theorem for $(M^{(1)}_{k,n},\; \ldots,\; M^{(l)}_{k,n})$ and for $(J^{(1)}_{k,n},\; \ldots,\; J^{(l)}_{k,n})$.  First, replacing integrals $\int_{1}^{\infty}.\frac{dx}{x}$ in Eq.\eqref{hillint} by $2\int_{1}^{\infty}.\log x \frac{dx}{x}$ and then removing the random centring yields the equivalent of Th.\ref{HillGaussDet} for $L_{k,n}$:

\begin{equation}
\label{Lrand}
\sqrt{k}\left(L_{k,n}^{(1)}- \frac{2}{\a^2},\dots, L_{k,n}^{(l)}- \frac{2}{\a^2}\right)\Rightarrow \mathcal{N}\left(0,\Theta\right).
\end{equation}
where $\Theta_{i,j}=\frac{8}{\a^3}\int_{1}^{\infty}\nu(x,1)\frac{dx}{x}+\frac{8}{\a^3}\int_{1}^{\infty}\nu(1,y)\frac{dy}{y}+8\frac{\nu(1,1)}{\a^4}$ and $\Theta_{i,i} = 20/\a^4$.
More generally
\begin{equation}
\label{LHillDet}
\sqrt{k}\left(I_{k,n}^{(1)}-\a_1,\dots, I_{k,n}^{(l)}-\a_l\right)\Rightarrow \mathcal{N}\left(0,\Omega_I\right),
\end{equation}
where, for any $1\leq i\leq l$, $I_{k,n}^{(i)}$ can be either $L_{k,n}^{(i)}$ with $\a_i = 2/\a^2$ or $H_{k,n}^{(i)}$ with $\a_i=1/\a$. Now, notice that $J_{k,n}$ and $L_{k,n}$ are functionals of $H_{k,n}$ and $L_{k,n}$ at first order. Indeed, 
\begin{align*}
J_{k,n}-\frac{1}{\a} & =  - \bigg(H_{k,n}-\frac{1}{\a}\bigg)+\frac{2}{\a}\left(L_{k,n}-\frac{2}{\a^2}\right) \\
                     & {}\quad +  o \left( \left(H_{k,n}-\frac{1}{\a}\right) + \left(L_{k,n}-\frac{2}{\a^2}\right) \right)\quad a.s.\\
M_{k,n}-\frac{1}{\a} & = \left(1-\frac{2}{\a}\right)\left(H_{k,n}-\frac{1}{\a}\right)+\frac{\a^2}{2}\left(L_{k,n}-\frac{2}{\a^2}\right)\\
                     & {}\quad +  o \left( \left(H_{k,n}-\frac{1}{\a}\right) + \left(L_{k,n}-\frac{2}{\a^2}\right) \right)\quad a.s.\\.
\end{align*}

The full generalisation of Theorem \ref{HillGaussDet} follows:

\begin{thm}
Under the assumptions of Theorem \ref{HillGaussDet} and Corollary \ref{corHillGauss}, we have
\begin{multline*}
\E\left[\sqrt{k}\left(M^{X}_{k_1,n}-\frac{1}{\a}\right)\sqrt{k}\left(M^{Y}_{k_2,n}-\frac{1}{\a}\right)\right]   \\
\xrightarrow[k\rightarrow\infty]{}   \int_{c_1^{1/\a}}^{\infty}\nu(x,c_2^{1/\a})\frac{dx}{x} + \int_{c_2^{1/\a}}^{\infty}\nu(c_1^{1/\a},y)\frac{dy}{y} + \frac{(\a-1)^2}{\a^2}\nu(c_1^{1/\a},c_2^{1/\a}).\\
\shoveleft{\E\left[\sqrt{k}\left(J^{X}_{k_1,n}-\frac{1}{\a}\right)\sqrt{k}\left(J^{Y}_{k_2,n}-\frac{1}{\a}\right)\right] }\\
\xrightarrow[k\rightarrow\infty]{}   \frac{1}{\a}\int_{c_1^{1/\a}}^{\infty}\nu(x,c_2^{1/\a})\frac{dx}{x} + \frac{1}{\a}\int_{c_2^{1/\a}}^{\infty}\nu(c_1^{1/\a},y)\frac{dy}{y}. \\
\shoveleft{\E\left[\sqrt{k}\left(H^{X}_{k_1,n}-\frac{1}{\a}\right)\sqrt{k}\left(M^{Y}_{k_2,n}-\frac{1}{\a}\right)\right]}\\
\xrightarrow[k\rightarrow\infty]{}   \int_{c_1^{1/\a}}^{\infty}\nu(x,c_2^{1/\a})\frac{dx}{x} + \frac{1-\a}{\a^2}\nu(c_1^{1/\a},c_2^{1/\a}).\\
\shoveleft{\E\left[\sqrt{k}\left(M^{X}_{k_1,n}-\frac{1}{\a}\right)\sqrt{k}\left(J^{Y}_{k_2,n}-\frac{1}{\a}\right)\right] }\\
\xrightarrow[k\rightarrow\infty]{}   \int_{c_1^{1/\a}}^{\infty}\nu(x,c_2^{1/\a})\frac{dx}{x} + \frac{1}{\a}\int_{c_2^{1/\a}}^{\infty}\nu(c_1^{1/\a},y)\frac{dy}{y}.
\end{multline*}
\begin{align*}
\E\left[\sqrt{k}\left(H^{X}_{k_1,n}-\frac{1}{\a}\right)\sqrt{k}\left(H^{Y}_{k_2,n}-\frac{1}{\a}\right)\right] &   \xrightarrow[k\rightarrow\infty]{}   \frac{\nu(c_1^{1/\a},c_2^{1/\a})}{\a^2}.\\
\E\left[\sqrt{k}\left(H^{X}_{k_1,n}-\frac{1}{\a}\right)\sqrt{k}\left(J^{Y}_{k_2,n}-\frac{1}{\a}\right)\right] &   \xrightarrow[k\rightarrow\infty]{}   \frac{1}{\a}\int_{c_1^{1/\a}}^{\infty}\nu(c_2^{1/\a},y)\frac{dy}{y}. \\
\end{align*}

\end{thm}

\bibliographystyle{plainnat}
\bibliography{bibliographieAbrev}

\end{document}